\documentclass[print,wide,eng]{draft}
\hypersetup{hidelinks}

\title{A Unified Kalman-Type Classification of Controllability\\
for Linear Stochastic Systems}
\author{Qi L\"u\thanks{School of Mathematics, Sichuan University, Chengdu 610064, China. E-mail: \texttt{lu@scu.edu.cn}.}
\and
Yu Wang\thanks{School of Mathematics, Southwest Jiaotong University, Chengdu 611756, China. E-mail: \texttt{yuwangmath@163.com}.}}
\date{}

\begin{document}
\maketitle
\begin{abstract}
We provide a unified algebraic classification of approximate controllability,
approximate null controllability, exact null controllability, and exact
controllability for finite-dimensional linear stochastic control systems with
constant coefficients and multidimensional Brownian noise. The controls may
enter both the drift and the diffusion terms, and the joint diffusion-control
matrix may have arbitrary rank. A finite decreasing sequence of subspaces
defines a common geometric obstruction. Its vanishing characterizes the two
approximate properties and exact null controllability. Exact controllability to
arbitrary random targets additionally requires the joint diffusion-control
matrix to be surjective. The codimension of the obstruction is given by a rank
difference of explicit polynomial block matrices, yielding finite tests in the
original coefficients. When the joint diffusion-control matrix is surjective,
we obtain a finite Kalman-type full-rank condition for exact controllability.
If the obstruction vanishes, a linear state feedback with a deterministic
time-dependent gain drives every initial state to zero at any prescribed positive time with finite energy and
polynomial small-time cost. An equivalent orthogonal elimination algorithm
computes the rank differences with polynomial storage and cost, avoiding
exponentially large block matrices.
A two-noise example relates the input classification to terminal errors
and control energy in penalized stochastic linear-quadratic control problems.
\end{abstract}



\section{Introduction}\label{sec:intro}

Linear stochastic systems with state- and control-dependent diffusion
provide a tractable setting for studying feedback design in the presence
of multiplicative uncertainty. Such models arise, for example, when
actuation signals are transmitted through uncertain input channels
\cite{NandanooriDiwadkarVaidya2018}. For multi-input systems, the organization
and allocation of these channels also enter the control design, as
illustrated by scheduling/control co-design over shared stochastic
channels \cite{ChenYaoQiu2019}. This model also connects mean-square
stabilization with generalized Riccati equations \cite{ZhangChen2004}.
These connections motivate a structural question for finite-horizon
control: which input configurations permit a prescribed terminal
requirement to be met by an adapted control with finite energy?
In particular, reaching the origin at a prescribed time and reaching
an arbitrary random terminal state impose different requirements on
how the control acts on the drift and the diffusion.

\subsection{Formulation of the controllability problems}

Let $W=(W^1,\dots,W^m)$ be an $m$-dimensional Brownian motion on a
complete probability space $(\Omega,\mathcal F,\mathbb P)$, and let
$\{\mathcal F_t\}_{t\geq0}$ be the usual augmentation of its natural
filtration. Write $\mathbb F$ for the progressive $\sigma$-algebra and
$\mathbb E$ for expectation.

For a Euclidean space $H$, let $L^2_{\mathcal F_t}(\Omega;H)$ denote
the square-integrable $\mathcal F_t$-measurable random variables.
For $T>0$, the space $L^2_{\mathbb F}(0,T;H)$ consists of progressively
measurable processes with $\mathbb E\int_0^T|v(t)|_H^2\,dt<\infty$.
We use $L^2_{\mathbb F}\bigl(\Omega;C([0,T];H)\bigr)$ for continuous adapted
processes with $\mathbb E\sup\limits_{0\leq t\leq T}|v(t)|_H^2<\infty$.

For $t\in[0,T]$, consider the linear control system
\begin{align}\label{eq:system}
\begin{cases}
\displaystyle dy(t)=[Ay(t)+Bu(t)]\,dt+\sum_{j=1}^m[C_jy(t)+D_ju(t)]\,dW^j(t),\\[4pt]
y(0)=y_0,
\end{cases}
\end{align}
where $A,C_j\in\mathbb R^{n\times n}$,
$B,D_j\in\mathbb R^{n\times r}$, $n,r,m\in \mathbb N$, $j=1,\dots,m$. The control
$u\in\mathcal U_T:=L^2_{\mathbb F}(0,T;\mathbb R^r)$.

By \cite[Theorem~3.2]{Lue2021a}, for every $y_0\in\mathbb R^n$ and
$u\in\mathcal U_T$, the control system \cref{eq:system} has a unique solution
$y(\cdot;y_0,u)\in L^2_{\mathbb F}\bigl(\Omega;C([0,T];\mathbb R^n)\bigr)$
satisfying
\begin{align}\label{eq:state-estimate}
\mathbb E\sup_{0\leq t\leq T}|y(t;y_0,u)|^2
\leq c(T)\Bigl(|y_0|^2+\mathbb E\int_0^T|u(t)|^2\,dt\Bigr).
\end{align}

Let us recall the following definition.
\begin{definition}\label{def:control}
For $T>0$, system \cref{eq:system} is said to be:
\begin{itemize}
\item (EC)  \emph{exactly controllable} if, for every
$y_0\in\mathbb R^n$ and
$\xi\in L^2_{\mathcal F_T}(\Omega;\mathbb R^n)$, there is
$u\in\mathcal U_T$ such that $y(T;y_0,u)=\xi$, $\mathbb P$-a.s.;
\item (AC)  \emph{approximately controllable} if, for every
$y_0\in\mathbb R^n$,
$\xi\in L^2_{\mathcal F_T}(\Omega;\mathbb R^n)$, and
$\varepsilon>0$, there is $u\in\mathcal U_T$ such that
\begin{align*}
\mathbb E|y(T;y_0,u)-\xi|^2<\varepsilon;
\end{align*}
\item (ENC) \emph{exactly null controllable} if, for every
$y_0\in\mathbb R^n$, there is $u\in\mathcal U_T$ such that
$y(T;y_0,u)=0$, $\mathbb P$-a.s.;
\item (ANC) \emph{approximately null controllable} if, for every
$y_0\in\mathbb R^n$ and $\varepsilon>0$, there is
$u\in\mathcal U_T$ such that
\begin{align*}
\mathbb E|y(T;y_0,u)|^2<\varepsilon.
\end{align*}
\end{itemize}
\end{definition}

Write $\mathcal C=[C_1^\top\ \cdots\ C_m^\top]^\top
\in\mathbb R^{mn\times n}$ and define the joint diffusion-control matrix
\begin{align}\label{10.1-eq1}
\mathcal D=[D_1^\top\ \cdots\ D_m^\top]^\top
\in\mathbb R^{mn\times r}.
\end{align}

For deterministic linear systems with constant coefficients, the Kalman
rank condition characterizes controllability, and exact, null and
approximate controllability coincide when the controls and targets
are deterministic.

The stochastic counterpart of this classical picture is less complete.
Even for linear stochastic differential equations with constant coefficients
and controls entering both the drift and diffusion, the deterministic
equivalences and rank condition do not carry over directly. For the class
of systems in \cref{eq:system}, two basic questions concern the relationships
among the controllability notions and their algebraic characterization.

The first concerns the relationships among the four properties in
\cref{def:control}. For scalar Brownian noise, Goreac
\cite[Theorem~3 and Remark~4]{Goreac} characterized approximate
controllability and approximate null controllability by a common
conditional-invariance obstruction. In contrast, approximate controllability
does not imply exact controllability to arbitrary random targets, as
illustrated by \cref{ex:singular-diffusion}. The question addressed here is whether approximate null
controllability also guarantees exact null control with finite energy
at every prescribed positive time for \cref{eq:system}.
This gap is not merely technical. The terminal control map takes values
in an infinite-dimensional space of random variables. Approximate null
controllability places the negative of the uncontrolled terminal state in
the closure of the range of the terminal control map, whereas exact null
controllability requires membership in the range itself. A sequence of
approximately null controls need not have uniformly bounded energy, so the
passage to exact null controllability is not automatic even in finite
dimensions. We also ask what additional condition is needed to attain arbitrary
random targets exactly.

The second question concerns a Kalman-type rank characterization in the
original coefficients. For scalar noise and a diffusion-control matrix
of full row rank, Peng
\cite[pp.~281--282 and Theorem~3.2]{Peng1994} obtained a necessary and
sufficient rank condition involving products of coefficient matrices
after an equivalent change of controls. L\"u and Zhang
\cite[Theorem~6.10]{Lue2021a} proved that products of length at most
$n-1$ suffice, and raised the problem of a rank characterization
directly for the original system in \cite[Remark~6.11]{Lue2021a}.
We seek a unified finite rank characterization in the original
coefficients for multidimensional noise and arbitrary rank of
$\mathcal D$, keeping the same adapted $L^2$ control class.

We resolve these two questions by identifying a common geometric
obstruction and expressing its codimension as the rank difference of
explicit polynomial block matrices in $A,B,(C_j,D_j)_{j=1}^m$.

From a design perspective, a criterion in the original coefficients
can be applied to candidate input configurations before constructing
a controller. Our results distinguish the feasibility of returning
the state exactly to the origin at a prescribed time from that of
attaining arbitrary square-integrable random terminal states. They
also reveal a concrete restriction on the input structure: every
controllability property considered here requires
$B\ker\mathcal D\neq\{0\}$, so some input combination must affect the
drift without adding diffusion. Exact controllability further requires
$\operatorname{rank}\mathcal D=mn$, and consequently $r\geq mn+1$;
see \cref{rem:input-structure}. The complete criteria in \cref{cor:rank-control} can thus test
feasibility before a minimum-energy steering problem is solved.
The numerical example in \cref{sec:numerical-example} illustrates
these distinctions through terminal-error and control-energy comparisons.
\subsection{Main results of this paper}\label{sec:geometry}

We first state the geometric classification and the feedback result, then
express the classification through finite rank conditions in the original
coefficient matrices.

We use $\langle\cdot,\cdot\rangle$ and $|\cdot|$ for Euclidean inner
products and norms, and $\lVert\cdot\rVert$ and $\lVert\cdot\rVert_F$ for
matrix operator and Frobenius norms. The standard coordinate vectors are
denoted by $e_1,\dots,e_n$. The transpose is denoted by $\top$, and
$\otimes$ denotes the Kronecker product. For a
subspace $V$, let $V^\perp$ be its Euclidean orthogonal complement and
$\Pi_V$ its orthogonal projector. Matrices with zero rows or columns are
allowed, with the usual rank, range and kernel conventions. In what follows, $c$ denotes a positive constant depending only on
$A,B,(C_j,D_j)_{j=1}^m$, and $c(T)$ may additionally depend on $T$;
these generic constants may change between formulas. Constants whose identity
is used subsequently are named separately.

To describe the obstruction to controllability, for $\eta\in L^2_{\mathcal F_T}(\Omega;\mathbb R^n)$,
consider the adjoint BSDE
\begin{align}\label{eq:adjoint}
\begin{cases}
\displaystyle dz(t)=-\Bigl(A^{\top}z(t)+\sum_{j=1}^mC_j^{\top}Z_j(t)\Bigr)\,dt
+\sum_{j=1}^m Z_j(t)\,dW^j(t),\\[4pt]
z(T)=\eta.
\end{cases}
\end{align}
Here $Z=(Z_1,\dots,Z_m)$. By \cite[Theorem~4.2]{Lue2021a},
\cref{eq:adjoint} has a unique solution with
$z\in L^2_{\mathbb F}\bigl(\Omega;C([0,T];\mathbb R^n)\bigr)$ and
$Z\in L^2_{\mathbb F}\bigl(0,T;(\mathbb R^n)^m\bigr)$.
Applying It\^o's formula to $\langle y,z\rangle$ and taking expectations
gives the duality identity
\begin{align}
\mathbb E\langle y(T),\eta\rangle-\langle y_0,z(0)\rangle
=\mathbb E\int_0^T
\Bigl\langle u,B^{\top}z+\sum_{j=1}^mD_j^{\top}Z_j\Bigr\rangle\,dt.
\label{eq:duality}
\end{align}
Thus $B^{\top}z+\sum\limits_{j=1}^mD_j^{\top}Z_j$ is the adjoint observation
associated with the control. An adjoint solution with zero observation is
undetected by every admissible control in \cref{eq:duality}.

For such an adjoint motion to remain in a subspace $V\subset\mathbb R^n$,
each martingale coefficient $w_j$ and the drift must belong to $V$.
We therefore define $\mathcal R(V)$ as the set of $v\in V$ for which
some $w_1,\dots,w_m\in V$ satisfy
\begin{align}\label{eq:map}
B^{\top}v+\sum_{j=1}^mD_j^{\top}w_j=0,\qquad
A^{\top}v+\sum_{j=1}^mC_j^{\top}w_j\in V.
\end{align}
Starting from $\mathcal V_0=\mathbb R^n$, set
\begin{align}\label{eq:flag}
	\mathcal V_{k+1}=\mathcal R(\mathcal V_k)\quad(k\geq0).
\end{align}
We denote $\mathcal V_n$ by $\mathcal V_*$.

The next lemma gives an intrinsic characterization of the limiting space.
\begin{lemma}\label{lem:flag}
	The sequence $(\mathcal V_k)_{k\geq0}$ decreases and is stationary from
	step $n$ onward. Its limit $\mathcal V_*$ is the largest subspace $V$
	for which some $\Gamma_1,\dots,\Gamma_m\in\mathbb R^{n\times n}$ satisfy
	\begin{align}\label{eq:invariance}
		\Gamma_jV\subset V\quad(j=1,\dots,m),\qquad
		\Bigl(A^{\top}+\sum_{j=1}^mC_j^{\top}\Gamma_j\Bigr)V\subset V,\qquad
		\Bigl(B^{\top}+\sum_{j=1}^mD_j^{\top}\Gamma_j\Bigr)V=\{0\}.
	\end{align}
\end{lemma}
\begin{proof}
For every subspace $V$, the linear conditions in \cref{eq:map} make
$\mathcal R(V)$ a subspace of $V$. Thus
$\mathcal V_{k+1}\subset\mathcal V_k$ for every $k\geq0$.
If equality holds at some step, then
the sequence is stationary from that step onward. Otherwise, the dimension
strictly decreases at each step; since the ambient space has dimension $n$,
stationarity must occur no later than step $n$.

Next we identify the limit $\mathcal V_*$. At this fixed point, for each
basis vector $v\in\mathcal V_*$, choose $w_1,\dots,w_m$ satisfying
\cref{eq:map}. Since the defining condition is linear, extending these
assignments linearly on $\mathcal V_*$ and then to $\mathbb R^n$ defines
matrices $\Gamma_1,\dots,\Gamma_m$ satisfying \cref{eq:invariance}
with $V=\mathcal V_*$.

Conversely, suppose that a subspace $V$ satisfies \cref{eq:invariance}
for some $\Gamma_1,\dots,\Gamma_m$. We show by induction on $k$ that
$V\subset\mathcal V_k$. For $k=0$, this holds since $\mathcal V_0$ is
the ambient space. Assume $V\subset\mathcal V_k$. For $v\in V$, set
$w_j=\Gamma_jv$. Then the invariance conditions ensure that
$v\in\mathcal V_{k+1}$, because the defining condition for
$\mathcal V_{k+1}$ is satisfied with these $w_j$. Hence
$V\subset\mathcal V_{k+1}$. Therefore $V\subset\mathcal V_k$ for all
$k$, and thus $V\subset\mathcal V_*$. This proves that $\mathcal V_*$
is the largest subspace with the stated property.
\end{proof}

Each step removes directions that fail the zero-observation and drift
conditions in \cref{eq:map}. At the limit, \cref{eq:invariance}
provides linear choices $w_j=\Gamma_jv$ in the obstruction space.

The following theorem classifies the four controllability properties through
the obstruction space and the joint diffusion rank.
\begin{theorem}\label{thm:geometry}
	For system \cref{eq:system}, the following statements hold.
	\begin{enumerate}[label=(\roman*)]
		\item For every $T>0$, approximate controllability, approximate null
		controllability, and exact null controllability are equivalent, and each
		holds if and only if $\mathcal V_*=\{0\}$.

		\item For every $T>0$, exact controllability to arbitrary targets in
		$ L^2_{\mathcal F_T}(\Omega;\mathbb R^n)$ holds if and only if
		\begin{align}\label{eq:exact-condition}
			\mathcal V_*=\{0\}\qquad\text{and}\qquad
			\operatorname{rank}\mathcal D=mn.
		\end{align}
	\end{enumerate}
\end{theorem}

In particular, each of the four controllability properties is independent
of the positive time horizon.

The next theorem realizes exact null controllability by a deterministic
linear feedback at each prescribed terminal time and gives a polynomial
small-time energy bound.
\begin{theorem}
	\label{thm:null-feedback}
	Assume that $\mathcal V_*=\{0\}$. Then for every $T>0$ there exists a
	deterministic, piecewise-constant matrix function
	$K_T:[0,T)\to\mathbb R^{r\times n}$, locally bounded on $[0,T)$, such
	that, for every $y_0\in\mathbb R^n$, the feedback control
	$u(t)=K_T(t)y(t)$ belongs to $\mathcal U_T$ and its state satisfies
	$y(T;y_0,u)=0$, $\mathbb P$-a.s. Moreover, there exist $c_0>0$ and an
	integer $N\geq1$, depending only on $A,B,(C_j,D_j)_{j=1}^m$, such that
	\begin{align}\label{eq:cost}
		\mathbb E\int_0^T|u(t)|^2\,dt\leq c_0T^{-N}|y_0|^2,
		\qquad 0<T\leq1.
	\end{align}
\end{theorem}

The gain $K_T$ may become unbounded as $t\uparrow T$, with switching
times accumulating at $T$. The energy bound concerns the control $u=K_Ty$.

\begin{table}[tb]
	\centering
	\caption{The only possible joint regimes under the standing assumptions.}
	\label{tab:classification-regimes}
	\begin{tabular}{@{}ccc@{}}
		\toprule
		$\mathcal V_*$ & $\operatorname{rank}\mathcal D$ & Properties\\
		\midrule
		$\neq\{0\}$ & arbitrary & none\\
		$\{0\}$ & $<mn$ & AC, ANC, ENC\\
		$\{0\}$ & $mn$ & AC, ANC, ENC, EC\\
		\bottomrule
	\end{tabular}
\end{table}

The following example illustrates the role of the joint diffusion rank in
\cref{tab:classification-regimes}.

\begin{example}\label{ex:singular-diffusion}
	Take $n=1$, $r=3$, $m=2$, $A=C_1=C_2=0$, and
	$B=(1,0,0)$, $D_1=(0,1,0)$, $D_2=(0,1,\varepsilon)$, where
	$\varepsilon\in\mathbb R$. Then \cref{eq:system} becomes
	\begin{align*}
	dy=u_1\,dt+u_2\,dW^1+(u_2+\varepsilon u_3)\,dW^2.
\end{align*}
	The zero-observation constraint is
	\begin{align*}
	B^{\top}v+D_1^{\top}w_1+D_2^{\top}w_2
	=\begin{pmatrix}v\\w_1+w_2\\\varepsilon w_2\end{pmatrix}=0.
\end{align*}
	Thus $\mathcal V_1=\{0\}$ for every $\varepsilon$. By
	\cref{thm:geometry}, AC, ANC and ENC hold at every $T>0$, whereas EC
	holds if and only if $\varepsilon\neq0$: the rank of $\mathcal D$ is
	one at $\varepsilon=0$ and two otherwise. Both $D_1$ and $D_2$, as
	well as their horizontal concatenation, have full row rank even at
	$\varepsilon=0$.
\end{example}

We next express the geometric classification through finite rank tests in
the given coefficient matrices, and obtain a Kalman-type full-rank
criterion for exact controllability when $\mathcal D$ is surjective.

Recall the definition of $\mathcal D$ in \cref{10.1-eq1}. Set
\begin{align}\label{eq:pair-start}
\mathsf K_1=B^{\top},\qquad
\mathsf N_1=\mathcal D^{\top}.
\end{align}
For $k\geq1$, write $\mathsf K_k\in\mathbb R^{a_k\times n}$ and
$\mathsf N_k\in\mathbb R^{a_k\times b_k}$, and form
\begin{align}\label{eq:pair}
	\mathsf K_{k+1}=
	\begin{pmatrix}
		B^{\top}\\
		0_{ma_k\times n}\\
		\mathsf K_kA^{\top}
	\end{pmatrix},\qquad
	\mathsf N_{k+1}=
	\begin{pmatrix}
		\mathcal D^{\top}&0&0\\
		I_m\otimes\mathsf K_k&I_m\otimes\mathsf N_k&0\\
		\mathsf K_k\mathcal C^{\top}&0&\mathsf N_k
	\end{pmatrix},
\end{align}
where $a_k=r\sum\limits_{\ell=0}^{k-1}(m+1)^\ell$ and
$b_k=mn\sum\limits_{\ell=0}^{k-1}(m+1)^\ell$.
The three column blocks of $\mathsf N_{k+1}$ have widths $mn$, $mb_k$
and $b_k$, and its three row blocks have heights $r$, $ma_k$ and
$a_k$. Put $\delta_0=0$ and, for $k\geq1$,
\begin{align}\label{eq:delta}
	\delta_k=\operatorname{rank}[\,\mathsf N_k\ \mathsf K_k\,]
	-\operatorname{rank}\mathsf N_k.
\end{align}
The block pair represents each obstruction space and determines its
codimension.

\begin{theorem}\label{thm:rank}
	For every $k\geq1$,
	\begin{align}\label{eq:rank-identity}
		\mathcal V_k
		=\bigl\{v\in\mathbb R^n:\exists\alpha\in\mathbb R^{b_k},\quad
		\mathsf K_kv+\mathsf N_k\alpha=0\bigr\},
	\end{align}
	and
	\begin{align}\label{eq:rank-codimension}
		\delta_k=n-\dim\mathcal V_k.
	\end{align}
\end{theorem}

At step $n$, this identity yields the controllability criterion.

\begin{corollary}\label{cor:rank-control}
	For every $T>0$, approximate controllability, approximate null
	controllability, and exact null controllability are each equivalent to
	\begin{align}\label{eq:rank-test}
		\operatorname{rank}[\mathsf N_n\ \mathsf K_n]
		-\operatorname{rank}\mathsf N_n=n.
	\end{align}
	Exact controllability is equivalent to \cref{eq:rank-test} together with
	$\operatorname{rank}\mathcal D=mn$.
\end{corollary}

\begin{proof}
	Apply \cref{thm:geometry,thm:rank} with $\mathcal V_*=\mathcal V_n$.
\end{proof}

When $\operatorname{rank}\mathcal D=mn$, let
$J=\mathcal D^\top(\mathcal D\mathcal D^\top)^{-1}$ and partition
$J=[J_1\ \cdots\ J_m]$, where $J_j\in\mathbb R^{r\times n}$. Define
$E_j=BJ_j$, $F=A-BJ\mathcal C$, and $H=B(I_r-J\mathcal D)$.
Set $H_0=H$ and, for $k\geq0$, define
\begin{align}\label{eq:word-blocks}
	H_{k+1}=[\,E_1H_k\ \cdots\ E_mH_k\ FH_k\,].
\end{align}
The matrix $H_k$ has $n$ rows and $r(m+1)^k$ columns, corresponding
to all products of exactly $k$ factors from $\{E_1,\dots,E_m,F\}$
followed by $H$. Thus the concatenation of $H_0,\dots,H_{n-1}$ has
$r\sum\limits_{k=0}^{n-1}(m+1)^k$ columns.

These matrices give a Kalman-type rank condition for exact controllability.

\begin{theorem}\label{prop:exact-rank}
	Assume that $\operatorname{rank}\mathcal D=mn$. At every $T>0$, system
	\cref{eq:system} is exactly controllable if and only if
	\begin{align}\label{eq:exact-rank}
		\operatorname{rank}[\,H_0\ H_1\ \cdots\ H_{n-1}\,]=n.
	\end{align}
\end{theorem}

\subsection{Relation to previous work}\label{intro:comparison}

The spaces in \cref{eq:map,eq:flag} extend Goreac's finite subspace
iteration for AC and ANC \cite[Theorem~3 and Remark~4]{Goreac} from
scalar to multidimensional Brownian noise, allowing arbitrary rank of
$\mathcal D$. \Cref{thm:geometry} further shows that the vanishing of
the same obstruction characterizes ENC at every prescribed $T>0$.

Guo, Yu and Zhang \cite{GYZ} characterize stabilizability through
approximate null controllability with cost: for each prescribed contraction
factor $\delta$, the control horizon may be chosen as $T=T(\delta)$.
In \cref{thm:null-feedback}, a linear state feedback with a deterministic
gain achieves exact null control at every prescribed $T>0$, with finite
energy and a polynomial small-time bound.

For scalar noise and a surjective diffusion-control matrix, Liu and Peng
\cite[Definition~8 and Theorem~9]{LiuPeng2002} prove that, for the
equivalent system with freely chosen diffusion, exact controllability is
equivalent to arbitrary exponential convergence-index assignment by linear
state feedback. Our classification also covers arbitrary joint diffusion rank,
where ENC need not imply EC.

S\^{\i}rbu and Tessitore \cite[Theorems~3.4 and~4.2]{ST} characterize
exact null controllability in Hilbert spaces by a singular Riccati equation
and a dual observability inequality, and obtain optimal null controls in
feedback form. On the common finite-dimensional, constant-coefficient
class, our result gives a finite algebraic test for this property. Their
model uses independent Brownian families for the state and control noises;
\cref{eq:system} also allows these terms to share the same Brownian channels.

For scalar noise and $\operatorname{rank}D_1=n$, the matrices $E_1,F,H$
defined before \cref{prop:exact-rank} express Peng's reduced criterion
\cite[Theorem~3.2]{Peng1994}, with the finite truncation of L\"u and
Zhang \cite[Theorem~6.10]{Lue2021a}, directly in $A,B,C_1,D_1$
without choosing a basis of $\ker D_1$. For the original system
considered in their treatment, L\"u and Zhang establish sufficiency but
leave necessity open, and ask in \cite[Remark~6.11]{Lue2021a} for a
rank characterization directly in the original coefficients.
\Cref{prop:exact-rank} gives the requested formulation and also covers
multidimensional noise.

Sun
\cite[Propositions~3.2--3.3 and Theorem~3.12]{Sun} gives finite-word,
Gramian, and Hautus characterizations for backward-structured systems with
freely chosen diffusion, where EC and ENC coincide.
The rank-difference criterion in \cref{cor:rank-control} additionally
covers nonsurjective $\mathcal D$, for which the diffusion cannot be
prescribed freely.
\Cref{sec:algorithm} evaluates this test by equivalent orthogonal
elimination with polynomial storage and arithmetic cost, without forming
the exponentially large block matrices.

Goreac, Li and Zhang \cite{GoreacLiZhang2026} study reduced-rank
mean-field systems by allowing weaker control integrability or
prescribing Gaussian terminal laws; our classification retains
$\mathcal U_T=L^2_{\mathbb F}(0,T;\mathbb R^r)$ and the four
controllability requirements in \cref{def:control} for the
constant-coefficient system \cref{eq:system}.

For broader background on stochastic controllability, we refer to
\cite{Zabczyk1981,EhrhardtKliemann1982} for early algebraic criteria,
\cite{BashirovMahmudov1999} for probabilistic controllability of
additive-noise systems, \cite{Mahmudov2001} for partial observation, and
\cite{BuckdahnQuincampoixTessitore2006} for invariant-subspace criteria with
drift-only control. Gramian and duality characterizations are developed in
\cite{Liu2010,WYYY}, while other terminal requirements include finite-chaos
approximation \cite{Dou2019}, exact output control \cite{WangYu2020}, and
partial controllability with terminal constraints \cite{Song2023}.

The remainder of the paper is organized as follows.
\Cref{sec:proof} proves the classification and feedback theorems.
\Cref{sec:rank} proves the rank criteria and discusses their
deterministic reduction and nonrobustness.
\Cref{sec:algorithm} presents the orthogonal elimination algorithm
and a numerical comparison of input configurations in penalized stochastic
linear-quadratic (LQ) control. \Cref{sec:conclusion} concludes the paper.
\section{Proofs of \texorpdfstring{\Cref{thm:geometry,thm:null-feedback}}{the classification and feedback theorems}}\label{sec:proof}

We first prove the equivalence of AC, ANC and $\mathcal V_*=\{0\}$ in
\Cref{sec:approximate}. We then prove \cref{thm:null-feedback}, which
supplies the implication to ENC needed for part~(i) of \Cref{thm:geometry}.

The feedback construction proceeds in three steps. In
\Cref{sec:strict-alternative}, we use approximate null controls to obtain
stabilizing feedbacks with any prescribed mean-square decay rate. In
\Cref{sec:feedback-growth}, a semialgebraic argument bounds their gains and
Lyapunov condition numbers polynomially in that rate. In
\Cref{sec:switching}, we apply these feedbacks on intervals accumulating at
$T$: the exponential state decay dominates the polynomial growth and makes
the control energy summable, establishing \Cref{thm:null-feedback}. We then
complete the proof of \Cref{thm:geometry} in \Cref{sec:switching}.

\subsection{Approximate controllability and the obstruction space}
\label{sec:approximate}

We prove the geometric characterization of AC and ANC using the
conditional-invariance approach of Goreac \cite[Theorem~3]{Goreac}.

\begin{lemma}\label{lem:approximate}
At every $T>0$, system \cref{eq:system} is approximately controllable if
and only if it is approximately null controllable, and these properties hold
if and only if $\mathcal V_*=\{0\}$.
\end{lemma}

\begin{proof}
Fix $T>0$ and define
$M_T:\mathcal U_T\to L^2_{\mathcal F_T}(\Omega;\mathbb R^n)$ by
$M_Tu=y(T;0,u)$. This linear map is bounded by
\cref{eq:state-estimate}, and \cref{eq:duality} gives
\begin{align*}
M_T^*\eta=B^{\top}z+\sum_{j=1}^mD_j^{\top}Z_j,
\end{align*}
where $(z,Z)$ solves \cref{eq:adjoint}. By linearity of the control system \cref{eq:system}, AC is equivalent to density of the range of $M_T$ in
$L^2_{\mathcal F_T}(\Omega;\mathbb R^n)$, or equivalently to
$\ker M_T^*=\{0\}$.

Suppose that $\mathcal V_*=\{0\}$ and $M_T^*\eta=0$. We show by
induction that $z(t)\in\mathcal V_k$ for all $t\in[0,T]$, $\mathbb P$-a.s.
For $k=0$ this is trivial. Assume it holds for some $k$. Since
$z(t)\in\mathcal V_k$, its projection onto $\mathcal V_k^\perp$
vanishes. Hence the projection of \cref{eq:adjoint} onto
$\mathcal V_k^\perp$ has zero drift and zero martingale part.
Consequently, the quadratic variation of the martingale part satisfies
\begin{align*}
\sum_{j=1}^m\int_0^T
\bigl|\Pi_{\mathcal V_k^\perp}Z_j(t)\bigr|^2\,dt=0,
\quad\mathbb P\text{-a.s.}
\end{align*}
Thus $Z_j(t)\in\mathcal V_k$ for every $j$, while the drift condition
is
\begin{align*}
A^{\top}z(t)+\sum_{j=1}^mC_j^{\top}Z_j(t)\in\mathcal V_k,
\end{align*}
almost everywhere with respect to the product measure
$dt\otimes d\mathbb P$. Together with
$
B^{\top}z+\sum\limits_{j=1}^mD_j^{\top}Z_j=0,
$
\cref{eq:map,eq:flag} give $z(t)\in\mathcal V_{k+1}$ almost everywhere.
Since $z$ is continuous and $\mathcal V_{k+1}$ is closed, this inclusion
holds for all $t\in[0,T]$, $\mathbb P$-a.s.  At $k=n$, we obtain $z=0$
and hence $\eta=z(T)=0$. Thus AC holds, and ANC follows by taking the
target to be zero.

Conversely, suppose that $\mathcal V_*\neq\{0\}$. Choose
$0\neq v\in\mathcal V_*$ and matrices $\Gamma_1,\dots,\Gamma_m$ as in
\cref{lem:flag}, and let $z$ solve
\begin{align*}
\begin{cases}
\displaystyle dz(t)=-\Bigl(A^{\top}+\sum_{j=1}^mC_j^{\top}\Gamma_j\Bigr)z(t)\,dt
+\sum_{j=1}^m\Gamma_jz(t)\,dW^j(t),\\[4pt]
z(0)=v.
\end{cases}
\end{align*}
The solution $z$ belongs to
$L^2_{\mathbb F}\bigl(\Omega;C([0,T];\mathbb R^n)\bigr)$ and remains in
$\mathcal V_*$ by \cref{eq:invariance}. Thus
\begin{align*}
Z=(\Gamma_1z,\dots,\Gamma_mz)
\in L^2_{\mathbb F}\bigl(0,T;(\mathbb R^n)^m\bigr)
\end{align*}
and $\eta=z(T)\in L^2_{\mathcal F_T}(\Omega;\mathbb R^n)$. The pair
$(z,Z)$ solves \cref{eq:adjoint} and satisfies
$B^{\top}z+\sum\limits_{j=1}^mD_j^{\top}Z_j=0$. Taking $y_0=v$ in
\cref{eq:duality} gives
\begin{align*}
\mathbb E\langle y(T;v,u),\eta\rangle=|v|^2
\qquad\text{for every }u\in\mathcal U_T.
\end{align*}
If ANC held, there would be controls $u_\ell\in\mathcal U_T$ with
$y(T;v,u_\ell)\to0$ in $L^2_{\mathcal F_T}(\Omega;\mathbb R^n)$.
The Cauchy--Schwarz inequality would then force the left-hand side to tend
to zero, contradicting $v\neq0$. Hence ANC fails, which proves the
converse implication.
\end{proof}

\subsection{Arbitrary-rate stabilizability}\label{sec:strict-alternative}

We first obtain stabilizing feedbacks with any prescribed mean-square decay
rate. We write $H>0$ (respectively, $H\geq 0$) when a symmetric
matrix $H$ is positive definite (respectively, positive semidefinite), and
use $<,\leq $ analogously.

For the $m\times m$ identity matrix $I_m$ and any matrix $M$, the
Kronecker product is
\begin{align*}
I_m\otimes M=
\begin{pmatrix}
M&0&\cdots&0\\
0&M&\ddots&\vdots\\
\vdots&\ddots&\ddots&0\\
0&\cdots&0&M
\end{pmatrix},
\end{align*}
with $m$ copies of $M$ on the diagonal.

For $\rho>0$, a symmetric $P\in\mathbb R^{n\times n}$, and
$U\in\mathbb R^{r\times n}$, put
\begin{align}\label{eq:lmi}
\mathcal M_\rho(P,U)=
\begin{pmatrix}
AP+PA^{\top}+BU+U^{\top}B^{\top}+\rho P & (\mathcal CP+\mathcal DU)^\top\\
\mathcal CP+\mathcal DU & -I_m\otimes P
\end{pmatrix}.
\end{align}
The matrix $\mathcal M_\rho(P,U)$ has size $(m+1)n$. A pair $(P,U)$
is called feasible at rate $\rho$ if $\mathcal M_\rho(P,U)<0$; its
lower-right block then gives $P>0$. Approximate null controllability
gives feasibility at every positive rate.

\begin{lemma}\label{lem:arbitrary-rate}
If $\mathcal V_*=\{0\}$, then for every $\rho>0$ there exist a symmetric
matrix $P\in\mathbb R^{n\times n}$ and a matrix
$U\in\mathbb R^{r\times n}$ such that
$\mathcal M_\rho(P,U)<0$.
\end{lemma}

\begin{proof}
\emph{Step 1: a linear family of approximate null controls.}
Fix $\rho>0$ and set $\delta=e^{-\rho}/2$, so that
$e^\rho\delta=1/2$. By \cref{lem:approximate}, for each coordinate vector
$e_i$ there is a control $u^{(i)}\in\mathcal U_1$ whose state
$y^{(i)}(t)=y(t;e_i,u^{(i)})$ satisfies
\begin{align*}
\mathbb E|y^{(i)}(1)|^2\leq\frac{\delta}{n},\qquad i=1,\dots,n.
\end{align*}
For $x=(x_1,\dots,x_n)^{\top}\in\mathbb R^n$, define
$u^x=\sum\limits_{i=1}^nx_iu^{(i)}$. By linearity, its state is
$y^x=\sum\limits_{i=1}^nx_iy^{(i)}$, and the Cauchy-Schwarz inequality gives
\begin{align}\label{eq:unit-contraction}
\mathbb E|y^x(1)|^2
\leq |x|^2\sum_{i=1}^n\mathbb E|y^{(i)}(1)|^2
\leq\delta|x|^2.
\end{align}
Each $u^{(i)}$ has finite energy, and \cref{eq:state-estimate} gives the
same property for $y^{(i)}$. Set
\begin{align*}
C_\rho:=\sum_{i=1}^n\mathbb E\int_0^1
\bigl(|y^{(i)}(s)|^2+|u^{(i)}(s)|^2\bigr)\,ds<\infty.
\end{align*}
Applying the Cauchy--Schwarz inequality to the sums defining $y^x$ and
$u^x$ and integrating in time gives
\begin{align}\label{eq:unit-energy}
\mathbb E\int_0^1\bigl(|y^x(s)|^2+|u^x(s)|^2\bigr)\,ds
\leq |x|^2\sum_{i=1}^n\mathbb E\int_0^1
\bigl(|y^{(i)}(s)|^2+|u^{(i)}(s)|^2\bigr)\,ds
=C_\rho|x|^2.
\end{align}
The controls $u^{(i)}$ are now fixed, so $C_\rho$ is independent of $x$.

\emph{Step 2: restarting the controls and obtaining finite weighted cost.}
The Wiener-space representation in
\cite[Propositions~9--10]{ClaisseTalayTan2016} allows us to write each
$u^{(i)}(s)$ as a nonanticipative measurable functional
$\Phi_i(s,W)$ of the path in $C([0,1];\mathbb R^m)$, up to
$ds\otimes d\mathbb P$-null sets. For an integer $j\geq0$, put
$W_j(s)=W(j+s)-W(j)$, $0\leq s\leq1$, and define
$u_j^{(i)}(s)=\Phi_i(s,W_j)$. Let $y_j^{(i)}$ be the corresponding
state starting from $e_i$ and driven by $W_j$. On $[j,j+1)$, use
\begin{align*}
u(j+s)=\sum_{i=1}^ny_i(j)u_j^{(i)}(s),\qquad 0\leq s<1,
\end{align*}
where $y_i(j)$ is the $i$th coordinate of the state already constructed
at time $j$. Linearity gives
\begin{align*}
y(j+s)=\sum_{i=1}^ny_i(j)y_j^{(i)}(s),\qquad 0\leq s\leq1.
\end{align*}
Conditionally on $\mathcal F_j$, the coefficients $y_i(j)$ are fixed,
whereas $W_j$ is a standard $m$-dimensional Brownian motion independent
of $\mathcal F_j$. The family $(y_j^{(i)},u_j^{(i)})_{i=1}^n$
consequently has the same joint law as $(y^{(i)},u^{(i)})_{i=1}^n$ and is
independent of $\mathcal F_j$. Thus \cref{eq:unit-contraction} gives
\begin{align}\label{eq:restart-contraction}
\mathbb E\bigl(|y(j+1)|^2\mid\mathcal F_j\bigr)\leq\delta|y(j)|^2.
\end{align}
Similarly, \cref{eq:unit-energy} gives
\begin{align}\label{eq:restart-energy}
\mathbb E\Bigl(\int_j^{j+1}\bigl(|y(t)|^2+|u(t)|^2\bigr)\,dt
\Bigm|\mathcal F_j\Bigr)\leq C_\rho|y(j)|^2.
\end{align}
This constructs a progressively measurable control with finite energy on
every bounded interval. Taking expectations in
\cref{eq:restart-contraction} and iterating gives
\begin{align*}
\mathbb E|y(j)|^2\leq\delta^j|y_0|^2.
\end{align*}
For the weighted cost on $[j,j+1]$, \cref{eq:restart-energy} now implies
\begin{align*}
\mathbb E\int_j^{j+1}e^{\rho t}\bigl(|y(t)|^2+|u(t)|^2\bigr)\,dt
\leq C_\rho e^{\rho(j+1)}\mathbb E|y(j)|^2
\leq C_\rho e^\rho(e^\rho\delta)^j|y_0|^2.
\end{align*}
Since $e^\rho\delta=1/2$, summing over $j\geq0$ gives
\begin{align}\label{eq:weighted-cost}
\mathbb E\int_0^\infty e^{\rho t}\bigl(|y(t)|^2+|u(t)|^2\bigr)\,dt
\leq C_\rho e^\rho\sum_{j=0}^\infty2^{-j}|y_0|^2
=2C_\rho e^\rho|y_0|^2.
\end{align}

\emph{Step 3: from finite weighted cost to the matrix inequality.}
Set $X(t)=e^{\rho t/2}y(t)$ and $v(t)=e^{\rho t/2}u(t)$. The
transformed system has drift matrix $A_\rho=A+\rho I_n/2$, and
\cref{eq:weighted-cost} bounds its unweighted infinite-horizon cost by
$2C_\rho e^\rho|y_0|^2$. We use finite-horizon value matrices to obtain a
stabilizing feedback, following the weighted LQ approach of
\cite{LiuPeng2002}.

For a symmetric matrix $H$, define
\begin{align*}
\mathsf R(H)=I_r+\sum_{j=1}^mD_j^\top HD_j
\end{align*}
and the state--control coefficient
\begin{align*}
\mathsf L(H)=B^\top H+\sum_{j=1}^mD_j^\top HC_j.
\end{align*}
Put
\begin{align*}
F_\rho(H)=A_\rho^\top H+HA_\rho+\sum_{j=1}^mC_j^\top HC_j+I_n
-\mathsf L(H)^\top\mathsf R(H)^{-1}\mathsf L(H).
\end{align*}
The map $F_\rho$ is defined wherever $\mathsf R(H)>0$. Consider
the Riccati equation in the horizon length $s$,
\begin{align}\label{eq:finite-riccati}
\begin{cases}
\dot{\mathsf H}_\rho(s)=F_\rho\bigl(\mathsf H_\rho(s)\bigr),\\[4pt]
\mathsf H_\rho(0)=0.
\end{cases}
\end{align}
With drift $A_\rho$, state and control weights $Q=I_n$ and $R=I_r$,
zero cross weight, and zero terminal weight, the standard stochastic
LQ theorem \cite[Chapter~6, Theorem~7.2, p.~320]{Yong1999}
applies. Its multidimensional version is given by equation~(7.60) and the
discussion on pp.~334--335 therein. The constant coefficients and
weights satisfy its boundedness, continuity and positivity assumptions.
After time reversal, it gives a unique nonnegative solution of
\cref{eq:finite-riccati} on every bounded interval. By uniqueness,
these solutions agree on overlapping intervals and define a global
solution satisfying
$\mathsf R\bigl(\mathsf H_\rho(s)\bigr)\geq I_r$ for $s\geq0$.
Moreover, $\dot{\mathsf H}_\rho(0)=I_n$.
Fix $s>0$ and set $H_t=\mathsf H_\rho(s-t)$.
For $X(0)=x$, It\^o's formula and completion of squares give
\begin{align*}
\mathbb E\int_0^s\bigl(|X(t)|^2+|v(t)|^2\bigr)\,dt
=x^\top\mathsf H_\rho(s)x+\mathbb E\int_0^s\bigl|
\mathsf R(H_t)^{1/2}v+
\mathsf R(H_t)^{-1/2}\mathsf L(H_t)X\bigr|^2\,dt.
\end{align*}
The last integral is nonnegative and vanishes for the feedback displayed
below. Consequently,
\begin{align*}
x^\top\mathsf H_\rho(s)x
=\inf_v\mathbb E\int_0^s\bigl(|X(t)|^2+|v(t)|^2\bigr)\,dt,
\end{align*}
where $X(0)=x$ and the infimum is over square-integrable progressively measurable
controls on $[0,s]$. It is attained by
\begin{align*}
v(t)=-\mathsf R\bigl(\mathsf H_\rho(s-t)\bigr)^{-1}
\mathsf L\bigl(\mathsf H_\rho(s-t)\bigr)X(t).
\end{align*}
Nonnegative running cost makes
$\mathsf H_\rho(s)$ nondecreasing in the Loewner order. Restricting the
control from \cref{eq:weighted-cost} to $[0,s]$ gives
\begin{align*}
0\leq \mathsf H_\rho(s)\leq 2C_\rho e^\rho I_n.
\end{align*}
Consequently, $\mathsf H_\rho(s)$ converges to a symmetric matrix
$H_\rho$. For fixed $h>0$, the integral form of
\cref{eq:finite-riccati} gives
\begin{align*}
\mathsf H_\rho(s+h)-\mathsf H_\rho(s)
=\int_0^hF_\rho\bigl(\mathsf H_\rho(s+\tau)\bigr)\,d\tau.
\end{align*}
Letting $s\to\infty$ and using continuity of $F_\rho$ gives
$F_\rho(H_\rho)=0$. Since $\dot{\mathsf H}_\rho(0)=I_n$ and the value
matrices are nondecreasing, $H_\rho>0$.

Set $K_\rho=-\mathsf R(H_\rho)^{-1}\mathsf L(H_\rho)$. The identity
$F_\rho(H_\rho)=0$ becomes
\begin{align}\label{eq:lq-lyapunov}
(A+BK_\rho)^\top H_\rho+H_\rho(A+BK_\rho)+\rho H_\rho
+\sum_{j=1}^m(C_j+D_jK_\rho)^\top H_\rho(C_j+D_jK_\rho)
=-I_n-K_\rho^\top K_\rho<0.
\end{align}
Take $P=H_\rho^{-1}$ and $U=K_\rho P$, so that
$(C_j+D_jK_\rho)P=C_jP+D_jU$. Multiplying
\cref{eq:lq-lyapunov} on both sides by $P$ gives
\begin{align*}
AP+PA^\top+BU+U^\top B^\top+\rho P
+(\mathcal CP+\mathcal DU)^\top
(I_m\otimes P^{-1})(\mathcal CP+\mathcal DU)<0.
\end{align*}
Since $P>0$, the Schur complement with lower-right block
$-I_m\otimes P$ yields $\mathcal M_\rho(P,U)<0$.
\end{proof}

\subsection{Polynomially bounded stabilizing feedbacks}
\label{sec:feedback-growth}

We now seek feasible pairs for which the feedback gain and the condition
number $\kappa(P)=\lambda_{\max}(P)/\lambda_{\min}(P)$ grow at most
polynomially with the decay rate.

For this purpose, we use semialgebraic sets, defined by finite Boolean
combinations of polynomial equalities and inequalities with real
coefficients. A function is semialgebraic if its graph is semialgebraic.
The Tarski--Seidenberg theorem permits projection and elimination of
real quantifiers \cite[Theorem~2.2.1 and Proposition~2.2.4]{BCR}.

\begin{lemma}\label{lem:tarski-seidenberg}
The image of a semialgebraic subset of $\mathbb R^{a+b}$ under
coordinate projection onto $\mathbb R^a$ is semialgebraic. More generally,
a set defined by polynomial equalities and inequalities, finite Boolean
operations and finitely many real existential or universal quantifiers
is semialgebraic.
\end{lemma}

We also use the polynomial growth bound for continuous semialgebraic
functions on closed semialgebraic sets
\cite[Proposition~2.6.2, p.~43]{BCR}, in the following form.

\begin{lemma}\label{lem:semialgebraic-growth}
If $f:[1,\infty)\to\mathbb R$ is continuous and semialgebraic, then
there exist $M\geq1$ and an integer $p\geq1$ such that
\begin{align*}
|f(\rho)|\leq M\rho^p,\qquad \rho\geq1.
\end{align*}
\end{lemma}

Related real-algebraic methods have been used in output-feedback and
stabilization problems \cite{AndersonBoseJury1975,FardigolaSheveleva2002}.
We apply these lemmas to the minimum of
$\lVert P\rVert_F^2+\lVert U\rVert_F^2$ over a normalized feasible set
for \cref{eq:lmi}. The normalization bounds $P^{-1}$, while the minimum
bounds both $\kappa(P)$ and $\lVert UP^{-1}\rVert^2$.

\begin{lemma}\label{lem:growth}
Suppose that a feasible pair exists for every $\rho\geq1$. There exist a
constant $M\geq1$ and an integer $p\geq1$, depending only on
$A,B,(C_j,D_j)_{j=1}^m$, such that for every $\rho\geq1$ one can choose
a feasible pair $\bigl(P(\rho),U(\rho)\bigr)$ satisfying
\begin{align}\label{eq:growth}
\kappa\bigl(P(\rho)\bigr)\leq M\rho^p
\end{align}
and
\begin{align}\label{eq:gain-growth}
\bigl\lVert U(\rho)P(\rho)^{-1}\bigr\rVert^2\leq M\rho^p.
\end{align}
\end{lemma}

\begin{proof}
\emph{Step 1: a semialgebraic minimum.}
For $\rho\geq1$, define the closed set
\begin{align*}
\mathscr F_\rho=
\bigl\{(P,U):P=P^{\top},\ \mathcal M_\rho(P,U)\leq -I_{(m+1)n}\bigr\}
\end{align*}
and put $g(P,U)=\lVert P\rVert_F^2+\lVert U\rVert_F^2$. By the
feasibility assumption and the homogeneity
\begin{align*}
\mathcal M_\rho(cP,cU)=c\mathcal M_\rho(P,U),\qquad c>0,
\end{align*}
$\mathscr F_\rho$ is nonempty. Its lower-right block gives
$I_m\otimes P\geq  I_{mn}$, hence $P\geq  I_n$. Since $g$ is
coercive, it attains the finite minimum
\begin{align*}
\mu(\rho)=\min_{(P,U)\in\mathscr F_\rho}g(P,U).
\end{align*}

We first show that $\mu$ is semialgebraic. The set of triples
$(\rho,P,U)$ with $\rho\geq1$ and $(P,U)\in\mathscr F_\rho$ is
semialgebraic: the matrix inequality is equivalent to nonnegativity of all
principal minors of $-\mathcal M_\rho(P,U)-I_{(m+1)n}$. For
$\rho\geq1$, a point $(\rho,s)$ belongs to the graph of $\mu$
precisely when some $(P,U)\in\mathscr F_\rho$ satisfies $g(P,U)=s$ and
no $(P',U')\in\mathscr F_\rho$ satisfies $g(P',U')<s$.
By \cref{lem:tarski-seidenberg}, this graph is semialgebraic.

\smallskip

\emph{Step 2: continuity of the minimum.}
Let $\rho_j\to\rho$ in $[1,\infty)$. To prove lower semicontinuity,
it suffices to consider $\liminf_j\mu(\rho_j)<\infty$. Pass to a
subsequence realizing this lower limit and choose minimizing pairs
$(P_j,U_j)\in\mathscr F_{\rho_j}$. Coercivity gives a further
subsequence converging to a pair $(P,U)$. Passing to the limit in the
matrix inequality gives $(P,U)\in\mathscr F_\rho$, and hence
\begin{align*}
\mu(\rho)\leq g(P,U)=\liminf_{j\to\infty}\mu(\rho_j).
\end{align*}
For upper semicontinuity, choose a minimizing pair $(P,U)$ at $\rho$
and fix $c>1$. For $\rho'$ near $\rho$,
\begin{align*}
\mathcal M_{\rho'}(cP,cU)
=c\mathcal M_\rho(P,U)
+c(\rho'-\rho)\begin{pmatrix}P&0\\0&0_{mn}\end{pmatrix}
\leq -cI_{(m+1)n}
+c(\rho'-\rho)\begin{pmatrix}P&0\\0&0_{mn}\end{pmatrix}.
\end{align*}
Thus $(cP,cU)\in\mathscr F_{\rho'}$ whenever $\rho'\geq1$ is
sufficiently close to $\rho$, so
\begin{align*}
\limsup_{\rho'\to\rho}\mu(\rho')\leq c^2\mu(\rho).
\end{align*}
Letting $c\downarrow1$ proves upper semicontinuity.

\smallskip

\emph{Step 3: polynomial bounds.}
Applying \cref{lem:semialgebraic-growth} to the continuous
semialgebraic function $\mu:[1,\infty)\to\mathbb R$, we obtain
a constant $M\geq1$ and an integer $p\geq1$ such that
\begin{align*}
\mu(\rho)\leq M\rho^p,\qquad \rho\geq1.
\end{align*}
For any minimizing pair $(\bar P,\bar U)$, we have $\bar P\geq  I_n$
and $\mu(\rho)\geq n\geq1$. In particular,
$\lVert\bar P^{-1}\rVert\leq1$, so
\begin{align*}
\kappa(\bar P)\leq\lVert\bar P\rVert_F
\leq\sqrt{\mu(\rho)}\leq\mu(\rho).
\end{align*}
The gain satisfies
\begin{align*}
\bigl\lVert\bar U\bar P^{-1}\bigr\rVert^2
\leq\lVert\bar U\rVert_F^2\leq\mu(\rho).
\end{align*}
Taking $P(\rho)=\bar P$ and $U(\rho)=\bar U$ proves
\cref{eq:growth,eq:gain-growth}.
\end{proof}

We next record the decay estimate that will be used on each switching
interval. For a feasible pair $(P,U)$ at rate $\rho$, set
$K=UP^{-1}$. The Schur complement of \cref{eq:lmi}, followed by
multiplication on both sides by $P^{-1}$, gives
\begin{align}\label{eq:lyapunov}
(A+BK)^{\top}P^{-1}+P^{-1}(A+BK)+\rho P^{-1}
+\sum_{j=1}^m(C_j+D_jK)^{\top}P^{-1}(C_j+D_jK)<0.
\end{align}
Define $V(x)=\langle P^{-1}x,x\rangle$ and apply the feedback $u=Ky$.
The initial state $y(t)$ may be any square-integrable
$\mathcal F_t$-measurable vector. It\^o's formula gives
\begin{align*}
dV(y(\tau))
&=2\langle P^{-1}y(\tau),(A+BK)y(\tau)\rangle\,d\tau
+\sum_{j=1}^m\bigl|P^{-1/2}(C_j+D_jK)y(\tau)\bigr|^2\,d\tau\\
&\quad+2\sum_{j=1}^m\langle P^{-1}y(\tau),
(C_j+D_jK)y(\tau)\rangle\,dW^j(\tau).
\end{align*}
By \cref{eq:lyapunov}, the drift is at most $-\rho V(y(\tau))$.
Multiplying by $e^{\rho(\tau-t)}$, localizing the stochastic integral,
and taking expectations yields
\begin{align*}
\mathbb EV(y(t+s))\leq e^{-\rho s}\mathbb EV(y(t)).
\end{align*}
Finally, the eigenvalue bounds
\begin{align*}
\frac{|x|^2}{\lambda_{\max}(P)}
\leq V(x)\leq\frac{|x|^2}{\lambda_{\min}(P)}
\end{align*}
give
\begin{align}\label{eq:decay}
\mathbb E|y(t+s)|^2\leq\kappa(P)e^{-\rho s}\mathbb E|y(t)|^2.
\end{align}
Both $K$ and $P$ are deterministic, so \cref{eq:decay} applies without
changing the feedback when $y(t)$ is random.

\subsection{Feedback construction and exact controllability}
\label{sec:switching}

\begin{proof}[Proof of \cref{thm:null-feedback}]
Assume $\mathcal V_*=\{0\}$. We first consider $0<T\leq1$ and treat
$T>1$ at the end of the construction. By
\cref{lem:arbitrary-rate,lem:growth}, we may choose feedbacks satisfying
\cref{eq:growth,eq:gain-growth} at every rate $\rho\geq1$. We will
shorten the switching intervals while increasing the decay rates fast
enough that the factor $e^{-\rho s}$ in \cref{eq:decay} compensates for
both polynomial factors in \cref{eq:growth,eq:gain-growth}.

\emph{Step 1: choosing the intervals and decay rates.}
Choose $L\geq\max(1,4p)$ so large that
\begin{align*}
4^{2p}ML^pe^{-L/2}\leq\frac12,
\end{align*}
which is possible because $e^{-L/2}$ decays faster than $L^p$. Set
\begin{align}\label{eq:switch}
t_j=T(1-2^{-j}),\qquad j\geq0.
\end{align}
The interval lengths are $\ell_j=t_{j+1}-t_j=T2^{-j-1}$. Choose
$\lambda_j=LT^{-2}4^j$ and
$K_j=U(\lambda_j)P(\lambda_j)^{-1}$, and define
\begin{align*}
K_T(t)=K_j,\qquad t\in[t_j,t_{j+1}),\quad j\geq0.
\end{align*}
In particular, $K_T(t_j)=K_j$ at every switching time.
On each interval, we solve the linear SDE with drift matrix
$A+BK_j$ and diffusion matrices $C_i+D_iK_j$, $i=1,\dots,m$,
using the terminal state from the preceding interval as its
initial state, starting with $y(0)=y_0$.
By \cite[Theorem~3.2, p.~132]{Lue2021a}, these equations have
unique square-integrable continuous adapted solutions.
They therefore join continuously at the switching times.
Since $K_T$ is bounded on every $[0,s]$ with $s<T$, this gives
a unique continuous adapted state on $[0,T)$, with
$u(t)=K_T(t)y(t)$ locally square-integrable.
Although
$\ell_j\to0$, the products $\lambda_j\ell_j=L2^{j-1}/T$ tend to
infinity. Writing $a_j=\mathbb E|y(t_j)|^2$,
\cref{eq:decay,eq:growth} give
\begin{align}\label{eq:state-step}
a_{j+1}\leq M\lambda_j^pe^{-\lambda_j\ell_j}a_j.
\end{align}

\emph{Step 2: summability of the control energy.}
The energy on $[t_j,t_{j+1}]$ contains the product of
$\lVert K_j\rVert^2$ and $\kappa\bigl(P(\lambda_j)\bigr)$. Indeed,
\cref{eq:decay,eq:growth,eq:gain-growth} imply
\begin{align}\label{eq:interval-energy}
\mathbb E\int_{t_j}^{t_{j+1}}|u(t)|^2\,dt
\leq\lVert K_j\rVert^2\kappa\bigl(P(\lambda_j)\bigr)a_j
\int_0^{\ell_j}e^{-\lambda_js}\,ds
\leq M^2\lambda_j^{2p}a_j\int_0^{\ell_j}e^{-\lambda_js}\,ds
\leq M^2\lambda_j^{2p-1}a_j.
\end{align}
This motivates the definition $b_j=\lambda_j^{2p}a_j$, for which
\cref{eq:interval-energy} is bounded by $M^2b_j/\lambda_j$. Since
$\lambda_{j+1}=4\lambda_j$, \cref{eq:state-step} gives
\begin{align}\label{eq:aj}
b_{j+1}
=\lambda_{j+1}^{2p}a_{j+1}
\leq 4^{2p}M\lambda_j^pe^{-\lambda_j\ell_j}b_j
\leq\frac12b_j.
\end{align}
To verify the last inequality, put $x_j=2^j/T\geq1$. Then
\begin{align*}
4^{2p}M\lambda_j^pe^{-\lambda_j\ell_j}
=4^{2p}ML^px_j^{2p}e^{-Lx_j/2}
\leq4^{2p}ML^pe^{-L/2}\leq\frac12.
\end{align*}
Here $x\mapsto x^{2p}e^{-Lx/2}$ is nonincreasing on $[1,\infty)$
because $L\geq4p$. It follows from \cref{eq:aj} that
$b_j\leq2^{-j}b_0$.  This, together with $\lambda_j=4^j\lambda_0$,
gives
\begin{align*}
\frac{b_j}{\lambda_j}\leq8^{-j}\frac{b_0}{\lambda_0}.
\end{align*}
Summing \cref{eq:interval-energy}, we therefore obtain
\begin{align}\label{eq:switching-energy}
\mathbb E\int_0^T|u(t)|^2\,dt
\leq M^2\sum_{j=0}^\infty\frac{b_j}{\lambda_j}
\leq\frac{M^2b_0}{\lambda_0}\sum_{j=0}^\infty8^{-j}
=\frac{8M^2b_0}{7\lambda_0}
=\frac87M^2L^{2p-1}T^{-(4p-2)}|y_0|^2.
\end{align}
In the last equality we used $b_0=\lambda_0^{2p}|y_0|^2$ and
$\lambda_0=LT^{-2}$.

\emph{Step 3: extending the state to the terminal time.}
So far, $(y,u)$ has been constructed on $[0,T)$ by solving the
feedback equation on successive intervals. Set $u(T)=0$; this
single endpoint value does not affect the integrals.
The resulting control is progressively measurable, and
\cref{eq:switching-energy} gives $u\in\mathcal U_T$.
Regarding $u$ as a fixed admissible control, we apply
\cite[Theorem~3.2, p.~132]{Lue2021a} to \cref{eq:system} on $[0,T]$.
The resulting solution belongs to
$L^2_{\mathbb F}\bigl(\Omega;C([0,T];\mathbb R^n)\bigr)$ and, by uniqueness,
agrees with the successively constructed state on every $[0,t_j]$.
We retain the notation $y$ for this extension.
Any other continuous adapted closed-loop solution agrees with $y$
on every $[0,t_j]$ by local uniqueness, and hence also at $T$
by continuity. Since $\lambda_j\geq1$, \cref{eq:aj} implies
\begin{align*}
a_j=\lambda_j^{-2p}b_j\leq2^{-j}b_0\longrightarrow0.
\end{align*}
In particular, $y$ is continuous in $L^2(\Omega;\mathbb R^n)$
at $T$. Consequently,
\begin{align*}
\mathbb E|y(T)|^2=\lim_{j\to\infty}\mathbb E|y(t_j)|^2
=\lim_{j\to\infty}a_j=0.
\end{align*}
The deterministic gain schedule is independent of $y_0$, so the
resulting null-control map is linear and bounded. Together with
\cref{eq:switching-energy}, the identity $y(T)=0$ proves the theorem for
$0<T\leq1$, with $N=4p-2$ in \cref{eq:cost}.

For $T>1$, let $\widehat K$ be the feedback schedule constructed above
for the unit horizon. Set $K_T(t)=0$ on $[0,T-1)$ and
$K_T(t)=\widehat K(t-T+1)$ on $[T-1,T)$. The uncontrolled state
$y(T-1)$ is square-integrable. Since \cref{eq:decay} applies to
square-integrable $\mathcal F_{T-1}$-measurable initial states, the
preceding estimates apply on the last unit interval with $|y_0|^2$
replaced by $\mathbb E|y(T-1)|^2$. They give finite control energy and
$y(T)=0$. This schedule is locally bounded on $[0,T)$ and independent
of $y_0$, completing the proof.
\end{proof}

\begin{proof}[Proof of \cref{thm:geometry}, part \textup{(i)}]
By \cref{lem:approximate}, AC and ANC are each equivalent to
$\mathcal V_*=\{0\}$. By \cref{thm:null-feedback}, this condition implies
ENC. Conversely, ENC implies ANC, completing the equivalence.
\end{proof}

\begin{proof}[Proof of \cref{thm:geometry}, part \textup{(ii)}]
Fix $T>0$. Exact controllability implies approximate null controllability
and hence $\mathcal V_*=\{0\}$. It also makes the map
$M_T:\mathcal U_T\to L^2_{\mathcal F_T}(\Omega;\mathbb R^n)$,
$M_Tu=y(T;0,u)$, surjective. Hilbert-space duality therefore gives a
constant $c(T)>0$ such that every adjoint solution satisfies
\begin{align}\label{eq:exact-observation}
\mathbb E|\eta|^2\leq c(T)\mathbb E\int_0^T
\Bigl|B^\top z(t)+\sum_{j=1}^mD_j^\top Z_j(t)\Bigr|^2\,dt.
\end{align}
We show that this requires $\operatorname{rank}\mathcal D=mn$.
Otherwise, choose
$q=[q_1^\top\ \cdots\ q_m^\top]^\top\in\ker\mathcal D^\top$ with
$|q|^2=\sum\limits_{j=1}^m|q_j|^2=1$. For $0<\varepsilon<\min(T,1)$,
prescribe
\begin{align*}
Z_j^\varepsilon(t)=\varepsilon^{-1/2}q_j
\mathbf1_{[T-\varepsilon,T]}(t),\qquad j=1,\dots,m,
\end{align*}
solve the adjoint differential equation forward from
$z^\varepsilon(0)=0$, and put $\eta_\varepsilon=z^\varepsilon(T)$.
By uniqueness, this is the adjoint solution for terminal value
$\eta_\varepsilon$. Set $g=\sum\limits_{j=1}^mC_j^\top q_j$. Then
$z^\varepsilon(t)=0$ for $t\leq T-\varepsilon$, and for later times
\begin{align*}
z^\varepsilon(t)
=-\varepsilon^{-1/2}\int_{T-\varepsilon}^t
e^{-A^\top(t-s)}g\,ds+\varepsilon^{-1/2}\sum_{j=1}^m\int_{T-\varepsilon}^t
e^{-A^\top(t-s)}q_j\,dW^j(s).
\end{align*}
The It\^o isometry gives
\begin{align*}
\mathbb E|\eta_\varepsilon|^2
=\frac1\varepsilon\sum_{j=1}^m\int_0^\varepsilon
|e^{-A^\top s}q_j|^2\,ds+\frac1\varepsilon\Bigl|\int_0^\varepsilon
e^{-A^\top s}g\,ds\Bigr|^2
\longrightarrow1.
\end{align*}
On the other hand, $\sum\limits_{j=1}^mD_j^\top Z_j^\varepsilon=0$.
Let $M_A=\max\limits_{0\leq s\leq1}\lVert e^{-A^\top s}\rVert$ and
$c_*=M_A^2(1+\sum\limits_{j=1}^m\lVert C_j\rVert^2)$.
For $\tau=t-(T-\varepsilon)\in[0,\varepsilon]$, the deterministic
term in the representation of $z^\varepsilon(t)$ has zero cross
expectation with the stochastic integral. The It\^o isometry
and the Cauchy--Schwarz inequality therefore give
\begin{align*}
\mathbb E|z^\varepsilon(t)|^2
=\frac1\varepsilon
\Bigl|\int_0^\tau e^{-A^\top s}g\,ds\Bigr|^2
+\frac1\varepsilon\sum_{j=1}^m
\int_0^\tau|e^{-A^\top s}q_j|^2\,ds
\leq\frac{M_A^2}{\varepsilon}
\Bigl(\tau^2|g|^2+\tau\sum_{j=1}^m|q_j|^2\Bigr)
\leq\frac{c_*}{\varepsilon}(\tau^2+\tau)\leq2c_*.
\end{align*}
Here we used $|g|^2\leq\sum\limits_{j=1}^m\lVert C_j\rVert^2$,
$\sum\limits_{j=1}^m|q_j|^2=1$, and $0\leq\tau\leq\varepsilon<1$.
In particular, $c_*$ is independent of $\varepsilon$.
Consequently,
\begin{align*}
\lim_{\varepsilon\to 0}\mathbb E\int_0^T
\Bigl|B^\top z^\varepsilon(t)
+\sum_{j=1}^mD_j^\top Z_j^\varepsilon(t)\Bigr|^2\,dt
=\lim_{\varepsilon\to 0}\mathbb E\int_{T-\varepsilon}^T
|B^\top z^\varepsilon(t)|^2\,dt
\leq\lim_{\varepsilon\to 0}2c_*\lVert B\rVert^2\varepsilon=0.
\end{align*}
Together with $\mathbb E|\eta_\varepsilon|^2\to1$, this contradicts
\cref{eq:exact-observation}.

Conversely, suppose that $\mathcal V_*=\{0\}$ and $\mathcal D$ is onto.
Set $J=\mathcal D^\top(\mathcal D\mathcal D^\top)^{-1}$. For
$\xi\in L^2_{\mathcal F_T}(\Omega;\mathbb R^n)$, let
$(x,\zeta_1,\dots,\zeta_m)$ solve
\begin{align*}
\begin{cases}
\displaystyle dx(t)=\bigl((A-BJ\mathcal C)x(t)+BJ\boldsymbol\zeta(t)\bigr)\,dt
+\sum_{j=1}^m\zeta_j(t)\,dW^j(t),\\[4pt]
x(T)=\xi.
\end{cases}
\end{align*}
where $\boldsymbol\zeta=[\zeta_1^\top\ \cdots\ \zeta_m^\top]^\top$.
By \cite[Theorem~4.2]{Lue2021a}, this BSDE has a square-integrable
solution. Define $\widehat u=J(\boldsymbol\zeta-\mathcal Cx)\in\mathcal U_T$.
Since $\mathcal DJ=I_{mn}$, its diffusion coefficients satisfy
\begin{align*}
\mathcal Cx+\mathcal D\widehat u=\boldsymbol\zeta.
\end{align*}
The drift coefficient is
\begin{align*}
Ax+B\widehat u=(A-BJ\mathcal C)x+BJ\boldsymbol\zeta.
\end{align*}
Thus $(x,\widehat u)$ solves \cref{eq:system} with terminal state $\xi$.
Since $\mathcal F_0$ is trivial modulo null sets, $x(0)$ is
deterministic. \Cref{thm:null-feedback} supplies a control $u^0$ whose
trajectory $y^0$ starts from $y_0-x(0)$ and satisfies $y^0(T)=0$.
By linearity, $y=x+y^0$ is the state corresponding to
$u=\widehat u+u^0$, with initial state
\begin{align*}
y(0)=x(0)+\bigl(y_0-x(0)\bigr)=y_0.
\end{align*}
Its terminal state is $y(T)=x(T)+y^0(T)=\xi$. This proves
\cref{eq:exact-condition}, and hence part~(ii).
\end{proof}

\begin{remark}\label{rem:terminal-interpretation}
In the terminology of \cite{LiuPeng2002}, exact terminal controllability
allows a deterministic initial state $x_\xi$ to be chosen together with
the control for each target
$\xi\in L^2_{\mathcal F_T}(\Omega;\mathbb R^n)$, rather than prescribing
the initial state. The BSDE construction above produces such an
$x_\xi=x(0)$ when $\mathcal D$ is onto. Exact null controllability then
corrects the difference $y_0-x_\xi$ without changing the terminal target.
Thus, by linearity, exact controllability is equivalent to exact terminal
controllability together with exact null controllability. Here $x_\xi$
is deterministic, not a random initial state depending on future noise.
\end{remark}

\section{Proofs of \texorpdfstring{\Cref{thm:rank,prop:exact-rank}}{the rank criteria}}\label{sec:rank}

We now prove the criteria in the original coefficients stated in
\cref{thm:rank,prop:exact-rank}.

\begin{proof}[Proof of \Cref{thm:rank}]
We prove \cref{eq:rank-identity} by induction. For $k=1$, the drift
condition in \cref{eq:map} is unrestricted, so \cref{eq:pair-start} gives
the representation. At step $k+1$, the conditions
$w_j\in\mathcal V_k$ are represented by
\begin{align}\label{10.3-eq1}
\mathsf K_kw_j+\mathsf N_k\alpha_j=0,\qquad j=1,\dots,m.
\end{align}
The drift constraint is
\begin{align}\label{10.3-eq2}
\mathsf K_k\Bigl(A^{\top}v+\sum_{j=1}^mC_j^{\top}w_j\Bigr)
+\mathsf N_k\beta=0.
\end{align}
Combining \cref{10.3-eq1,10.3-eq2} with
$B^{\top}v+\sum\limits_{j=1}^mD_j^{\top}w_j=0$, we obtain
\begin{align*}
\mathsf K_{k+1}v+\mathsf N_{k+1}
\bigl[w_1^\top\ \cdots\ w_m^\top\ \alpha_1^\top\ \cdots\ \alpha_m^\top\ \beta^\top\bigr]^\top=0,
\end{align*}
with the block matrices from \cref{eq:pair}.

Conversely, the block equation imposes the zero-observation constraint and
places every $w_j$ and the drift
$A^{\top}v+\sum\limits_{j=1}^mC_j^{\top}w_j$ in $\mathcal V_k$.
These conditions also imply $v\in\mathcal V_k$. Indeed, starting with
$v\in\mathcal V_0$, if $v\in\mathcal V_\ell$ for $0\leq\ell<k$,
then $\mathcal V_k\subset\mathcal V_\ell$ and \cref{eq:map,eq:flag}
give $v\in\mathcal V_{\ell+1}$. The same constraints therefore yield
$v\in\mathcal V_{k+1}$, proving \cref{eq:rank-identity}.

For matrices $K,N$ with the same number of rows, the linear map
$v\mapsto Kv+\operatorname{im}N$ into the quotient by
$\operatorname{im}N$ has kernel
$\{v:Kv\in\operatorname{im}N\}$ and rank
$(\operatorname{rank}[N\ K]-\operatorname{rank}N)$. This proves
\cref{eq:rank-codimension}.
\end{proof}

Every block is polynomial in $A,B,(C_j,D_j)_{j=1}^m$, so the criterion
does not require computing the obstruction spaces or inverting a coefficient
matrix.

\begin{remark}\label{rem:deterministic-reduction}
When the diffusion coefficients vanish, \cref{cor:rank-control} recovers
the classical Kalman rank condition. The distinction between deterministic
and random terminal targets explains why EC does not become equivalent to
the other three properties in this case.

Indeed, if $C_j=D_j=0$ for every $j$, the recursion in \cref{eq:map}
gives
\begin{align*}
\mathcal V_k=\bigcap_{\ell=0}^{k-1}
\ker\bigl\{B^{\top}(A^{\top})^\ell\bigr\}
\qquad(k\geq1).
\end{align*}
Hence
\begin{align}\label{eq:kalman-reduction}
\delta_n=\operatorname{rank}[\,B\ AB\ \cdots\ A^{n-1}B\,].
\end{align}
Consequently, \cref{eq:rank-test} becomes the Kalman rank condition for AC, ANC and
ENC. For the classical problem with deterministic controls and terminal
targets, this condition is equivalent to all four controllability
properties. Here controls remain in $\mathcal U_T$ and terminal targets
range over $L^2_{\mathcal F_T}(\Omega;\mathbb R^n)$: adapted controls can
produce random terminal states even when the diffusion vanishes. Exact
controllability to all such targets fails because $\mathcal D=0$ violates
the necessary condition in \cref{cor:rank-control}.

The following construction explains how adapted controls approximate
arbitrary random terminal targets even when the diffusion vanishes. Assume
the Kalman condition. For
$\xi\in L^2_{\mathcal F_T}(\Omega;\mathbb R^n)$ and $0<h<T$, set
$\xi_h=\mathbb E[\xi\mid\mathcal F_{T-h}]$. The target $\xi_h$ is
already known at time $T-h$. After zero control on $[0,T-h]$, the
deterministic Gramian formula gives a control on $[T-h,T]$ reaching
$\xi_h$ exactly. This control is adapted and has finite expected energy
for each $h$, since $\xi_h$ is $\mathcal F_{T-h}$-measurable and
square-integrable. Under the Brownian natural filtration,
$\xi_h\to\xi$ in $L^2_{\mathcal F_T}(\Omega;\mathbb R^n)$ as
$h\downarrow0$, which gives AC. The corresponding control energies need
not remain bounded.
\end{remark}

\begin{remark}\label{rem:input-structure}
The first iterate is
\begin{align}\label{eq:first-step}
\mathcal V_1=(B\ker\mathcal D)^\perp.
\end{align}
If $B\ker\mathcal D=\{0\}$, the iteration is stationary from the start.
Hence each of the four controllability properties requires
$B\ker\mathcal D\neq\{0\}$. Exact controllability also requires
$\operatorname{rank}\mathcal D=mn$, so the control dimension satisfies
$r\geq mn+1$. Thus, besides enough input directions to prescribe all
diffusion channels, at least one input combination that vanishes in every
diffusion channel must act nontrivially through the drift.
\end{remark}

The following example, adapted from
\cite[Section~6.3 and Remark~6.24]{Lue2021a}, shows that arbitrarily small diffusion perturbations
can destroy AC, ANC and ENC, ruling out a full-rank characterization of
these properties by a fixed-size matrix depending continuously on the
coefficients.

\begin{example}\label{rem:nonrobustness}
For fixed $m\geq1$, take $n=2$, $r=1$ and consider
\begin{align*}
\begin{cases}
dy_1=y_2\,dt+\varepsilon y_2\,dW^1(t),\\[4pt]
dy_2=u\,dt.
\end{cases}
\end{align*}
Its drift coefficients are $B=(0,1)^\top$ and
\begin{align*}
A=\begin{pmatrix}0&1\\0&0\end{pmatrix}.
\end{align*}
The diffusion coefficients are $C_1=\varepsilon A$, $C_j=0$ for
$j\geq2$, and $D_j=0$ for every $j$. Since $B^\top v=v_2$, the
first step of \cref{eq:flag} gives
$\mathcal V_1=\operatorname{span}\{e_1\}$ for every $\varepsilon$.
At the second step, write $v=ae_1$ and $w_j=b_je_1$. The remaining
constraint is
\begin{align*}
A^\top v+\sum_{j=1}^mC_j^\top w_j=(a+\varepsilon b_1)e_2
\in\operatorname{span}\{e_1\},
\end{align*}
or equivalently $a+\varepsilon b_1=0$. If $\varepsilon=0$, this forces
$a=0$; if $\varepsilon\neq0$, every $a$ is admissible by choosing
$b_1=-a/\varepsilon$ and $b_j=0$ for $j\geq2$. Since $n=2$, it
follows that
\begin{align*}
\mathcal V_*=\mathcal V_2=
\begin{cases}
\{0\},&\varepsilon=0,\\
\operatorname{span}\{e_1\},&\varepsilon\neq0.
\end{cases}
\end{align*}

By \cref{thm:geometry}, AC, ANC and ENC hold at every $T>0$ when
$\varepsilon=0$, and all three fail at every $T>0$ when
$\varepsilon\neq0$. Exact controllability fails throughout this family
because $\mathcal D=0$.

For fixed $n,r,m$, maximal rank of a fixed-size matrix depending
continuously on the coefficient matrices is an open condition. The family
above therefore precludes such a full-rank characterization of AC, ANC or
ENC in general. For the rank-difference criterion,
\cref{eq:rank-codimension} gives $\delta_2=2$ when $\varepsilon=0$ and
$\delta_2=1$ otherwise. Thus the condition $\delta_2=n$ in
\cref{eq:rank-test} is not open along this family.
\end{example}

\begin{proof}[Proof of \Cref{prop:exact-rank}]
Put $w=[w_1^\top\ \cdots\ w_m^\top]^\top$.
Since $J\mathcal D$ is the orthogonal projector onto
$\operatorname{im}\mathcal D^\top$, we have
$(I_r-J\mathcal D)\mathcal D^\top=0$ and
$J^\top\mathcal D^\top=I_{mn}$. Multiplying
$B^\top v+\mathcal D^\top w=0$ by $I_r-J\mathcal D$ and $J^\top$
therefore gives
\begin{align*}
H^\top v=0,\qquad w=-J^\top B^\top v.
\end{align*}
Conversely, these conditions imply
\begin{align*}
B^\top v+\mathcal D^\top w
=(I_r-\mathcal D^\top J^\top)B^\top v=H^\top v=0.
\end{align*}
The second condition reads $w_j=-E_j^\top v$ for every $j$, and the
drift becomes
\begin{align*}
A^\top v+\sum_{j=1}^mC_j^\top w_j
=(A^\top-\mathcal C^\top J^\top B^\top)v=F^\top v.
\end{align*}
By \cref{lem:flag}, $\mathcal V_*$ is therefore the largest common
$(E_1^\top,\dots,E_m^\top,F^\top)$-invariant subspace contained in
$\ker H^\top$.

Put $S_k=\operatorname{im}[\,H_0\ H_1\ \cdots\ H_k\,]$.
By \cref{eq:word-blocks},
\begin{align*}
S_{k+1}=S_k+\sum_{j=1}^mE_jS_k+FS_k.
\end{align*}
Thus the sequence is increasing, and an equality $S_{k+1}=S_k$
implies invariance under every $E_j$ and $F$, so all subsequent
spaces coincide. If $H=0$, all these spaces are zero.
Otherwise $\dim S_0\geq1$; if no equality occurs before $S_{n-1}$,
then $\dim S_{n-1}\geq\dim S_0+n-1\geq n$.
In either case, $S:=S_{n-1}$ is invariant under all these matrices.
If $U$ is any other common invariant subspace containing
$\operatorname{im}H$, the recursion gives $S_k\subset U$ by
induction. Hence $S$ is the smallest such subspace.

For $v\in S^\perp$, $x\in S$, and
$M\in\{E_1,\dots,E_m,F\}$, we have
\begin{align*}
\langle M^\top v,x\rangle=\langle v,Mx\rangle=0.
\end{align*}
Consequently, $S^\perp$ is invariant under
$E_1^\top,\dots,E_m^\top,F^\top$. Since $\operatorname{im}H\subset S$,
it is also contained in $\ker H^\top$.
The maximality of $\mathcal V_*$ therefore gives
$S^\perp\subset\mathcal V_*$.
Conversely, the same inner-product identity shows that
$\mathcal V_*^\perp$ is invariant under every $E_j$ and $F$,
and $\mathcal V_*\subset\ker H^\top$ implies
$\operatorname{im}H\subset\mathcal V_*^\perp$.
The minimality of $S$ gives $S\subset\mathcal V_*^\perp$,
and hence $\mathcal V_*\subset S^\perp$.
We conclude that $S_{n-1}^\perp=\mathcal V_*$.
Thus \cref{eq:exact-rank} is equivalent to $\mathcal V_*=\{0\}$,
and the conclusion follows from \cref{thm:geometry}.
\end{proof}

\begin{remark}\label{rem:surjective-diffusion}
Since $I_r-J\mathcal D$ is the orthogonal projector onto
$\ker\mathcal D$,
\begin{align*}
\operatorname{im}H=B\ker\mathcal D.
\end{align*}
Thus $H$ represents the drift-control directions left after the diffusion
has been prescribed, without choosing a basis of $\ker\mathcal D$. For
$m=1$, the matrices $E_1,F,H$ express Peng's reduced word criterion
\cite[Theorem~3.2]{Peng1994} directly in the original
coefficients, with the finite truncation in
\cite[Theorem~6.10]{Lue2021a}. Related finite-word characterizations for
backward-structured systems are given in
\cite[Propositions~3.2--3.3]{Sun}.
\end{remark}

The block dimensions in \cref{eq:pair} grow exponentially.
\Cref{sec:algorithm} gives an equivalent reduced computation.

\section{Computational method and numerical illustration}\label{sec:algorithm}

The rank differences in \cref{eq:delta} can be computed by orthogonal
elimination on the current obstruction space. We first give the exact
algebraic algorithm and prove its equivalence to the subspace iteration.
We then describe a QR/SVD realization with polynomial storage and arithmetic
cost, avoiding the exponentially large matrices in \cref{eq:pair}.

\subsection{Orthogonal elimination}

Suppose $[\,Q_k\ V_k\,]$ is orthogonal, with
$\operatorname{im}V_k=\mathcal V_k$ and
$\operatorname{im}Q_k=\mathcal V_k^\perp$. Write
$q_k=\dim\mathcal V_k^\perp$ and $d_k=n-q_k$, so that
$Q_k\in\mathbb R^{n\times q_k}$ and $V_k\in\mathbb R^{n\times d_k}$.
Initially, $Q_0$ has no columns and $V_0=I_n$. By \cref{lem:flag},
$\mathcal V_{k+1}\subset\mathcal V_k$, so a candidate for the next space
has the form $v=V_kb$, with $b\in\mathbb R^{d_k}$. Write the auxiliary
vectors as $w_j=V_ka_j$, with $a_j\in\mathbb R^{d_k}$, and put
$a=[a_1^\top\ \cdots\ a_m^\top]^\top$. Since a vector belongs to
$\mathcal V_k$ precisely when $Q_k^\top$ annihilates it, the input and
drift constraints become
\begin{align*}
B^\top V_kb+\sum_{j=1}^mD_j^\top V_ka_j=0,\qquad
Q_k^\top A^\top V_kb+\sum_{j=1}^mQ_k^\top C_j^\top V_ka_j=0.
\end{align*}
The reduced block multiplying $b$ is
\begin{align}\label{eq:reduced-blocks}
L_k=
\begin{pmatrix}
	B^\top V_k\\
	Q_k^\top A^\top V_k
\end{pmatrix}
\in\mathbb R^{(r+q_k)\times d_k},
\end{align}
and the block multiplying $a$ is
\begin{align}\label{eq:elimination-block}
M_k=
\begin{pmatrix}
	D_1^\top V_k&\cdots&D_m^\top V_k\\
	Q_k^\top C_1^\top V_k&\cdots&Q_k^\top C_m^\top V_k
\end{pmatrix}
\in\mathbb R^{(r+q_k)\times(md_k)}.
\end{align}
The constraints read
\begin{align}\label{eq:reduced-constraint}
L_kb+M_ka=0.
\end{align}
The coordinates of $a$ already impose $w_j\in\mathcal V_k$ for every
$j$, so these membership conditions require no additional equations.

To eliminate $a$ from \cref{eq:reduced-constraint}, let $U_{0,k}$ have
orthonormal columns spanning $\ker M_k^\top$ and set
\begin{align}\label{eq:residual}
R_k=U_{0,k}^\top L_k.
\end{align}
Because $(\ker M_k^\top)^\perp=\operatorname{im}M_k$, eliminating $a$ from \cref{eq:reduced-constraint} gives
\begin{align*}
V_kb\in\mathcal V_{k+1}
\Longleftrightarrow L_kb\in\operatorname{im}M_k
\Longleftrightarrow U_{0,k}^\top L_kb=0
\Longleftrightarrow b\in\ker R_k.
\end{align*}
Thus $\mathcal V_{k+1}=V_k\ker R_k$. Let
$\Delta q_k=\operatorname{rank}R_k$, and choose an orthogonal matrix
$O_k=[\,O_{1,k}\ O_{0,k}\,]$ whose first $\Delta q_k$ columns span
$(\ker R_k)^\perp$ and whose remaining columns span $\ker R_k$. The
update is
\begin{align}\label{eq:update}
Q_{k+1}=[\,Q_k\ V_kO_{1,k}\,],\qquad
V_{k+1}=V_kO_{0,k}.
\end{align}
The columns of $V_kO_{1,k}$ are the directions removed at this stage;
they are appended to the orthogonal complement already computed.

This construction computes exactly the rank differences of the prescribed
block pair.

\begin{proposition}\label{prop:staircase}
In exact arithmetic, the updates in
\cref{eq:reduced-blocks,eq:elimination-block,eq:residual,eq:update}
preserve orthogonality of $[Q_k\ V_k]$ and satisfy
$\operatorname{im}V_k=\mathcal V_k$, $q_k=\delta_k$, and
$\Delta q_k=\delta_{k+1}-\delta_k$. The iteration reaches a stationary
space or the zero space after at most $n$ stages, and returns the
classification in \cref{tab:classification-regimes} after checking
$\operatorname{rank}\mathcal D$.
\end{proposition}

\begin{proof}
The elimination above gives
$\operatorname{im}V_{k+1}=V_k\ker R_k=\mathcal V_{k+1}$. Since $V_k$
is an isometry on $\mathbb R^{d_k}$, the two blocks $V_kO_{1,k}$ and
$V_kO_{0,k}$ are orthonormal and mutually orthogonal. Both are orthogonal
to $Q_k$, so \cref{eq:update} preserves a complete orthogonal state basis
and gives $q_{k+1}=q_k+\Delta q_k$. Starting from
$q_0=\delta_0=0$, \cref{eq:rank-codimension} yields $q_k=\delta_k$ and
the asserted increment identity.

If $\Delta q_k=0$, then $\ker R_k=\mathbb R^{d_k}$ and
$\mathcal V_{k+1}=\mathcal V_k$; by \cref{lem:flag}, all subsequent
spaces coincide. Otherwise the dimension drops by at least one. Hence the
algorithm terminates within $n$ stages. The returned space is
$\mathcal V_*$, and \cref{thm:geometry} gives the classification.
\end{proof}

\Cref{alg:staircase} summarizes the exact-arithmetic computation.

\begin{algorithm}[htbp]
\caption{Orthogonal elimination for the controllability criterion (exact arithmetic)}
\label{alg:staircase}
\begin{algorithmic}[1]
	\REQUIRE $A,C_j\in\mathbb R^{n\times n}$, $B,D_j\in\mathbb R^{n\times r}$,
	$j=1,\dots,m$.
	\STATE $Q\gets[\,]$, $V\gets I_n$; form
	$\mathcal D=[D_1^\top\ \cdots\ D_m^\top]^\top$.
	\FOR{$k=0,\dots,n-1$}
	\STATE Form $L$ and $M$ from $Q,V$ by
	\cref{eq:reduced-blocks,eq:elimination-block}.
	\STATE Reduce $M$ by left orthogonal transformations and apply the same
	transformations to $L$.
	\STATE Let $R$ consist of the rows of the transformed $L$ corresponding
	to the zero rows of the transformed $M$.
	\STATE Find $\Delta q=\operatorname{rank}R$ and an orthogonal matrix
	$O=[O_1\ O_0]$, with $\operatorname{im}O_0=\ker R$.
	\IF{$\Delta q=0$}
	\STATE \textbf{return} nonzero subspace $\operatorname{im}V$; all four
	controllability properties fail.
	\ENDIF
	\STATE $Q\gets[Q\ VO_1]$, $V\gets VO_0$ (using the old $V$ on
	both right sides).
	\IF{$V$ has no columns}
	\STATE \textbf{return} AC, ANC and ENC hold; EC holds if and only if
	$\operatorname{rank}\mathcal D=mn$.
	\ENDIF
	\ENDFOR
\end{algorithmic}
\end{algorithm}

\subsection{Numerical realization and complexity}

The elimination uses Householder QR and SVD \cite{LAPACK}.
At one stage, put $p=r+q$, $d=n-q$, $h=md$ and
$s=\operatorname{rank}M$, so $M\in\mathbb R^{p\times h}$ and
$L\in\mathbb R^{p\times d}$. When $p\geq h$, apply Householder QR
to $[M\ L]$, then an SVD to the leading $h\times h$ triangular
block of the transformed $M$. Apply the same left transformations
to $L$ and retain all $p-s$ rows corresponding to zero rows of $M$.
These form $R$. Reflectors avoid storing a full $p\times p$ factor.
When $p<h$, a full left singular-vector matrix of $M$ supplies
these rows directly; its full right factor is unnecessary.

An SVD of $R$ gives the decomposition in \cref{eq:update}.
Its full $d\times d$ right factor is required when $R$ is wide,
since a thin SVD omits part of the nullspace. If $R$ has no rows,
its kernel is $\mathbb R^d$ and the iteration stops. After updating,
Householder QR may reorthogonalize $[Q\ V]$ in its existing order.

The use of the residual itself is important for numerical rank decisions.
Indeed,
\begin{align*}
R_k^\top R_k
=L_k^\top\bigl(I-\Pi_{\operatorname{im}M_k}\bigr)L_k.
\end{align*}
This Gram matrix has the correct kernel, but its nonzero eigenvalues are
the squares of the nonzero singular values of $R_k$. We therefore compute
the singular values of $R_k$ directly. We also avoid forming the projected
residual by the subtraction $L_k-U_{1,k}(U_{1,k}^\top L_k)$, where
$U_{1,k}$ spans $\operatorname{im}M_k$ orthonormally: cancellation can
obscure a small residual. Applying the orthogonal transformations to $L_k$
and extracting the required rows gives \cref{eq:residual} without this
subtraction.

In floating-point arithmetic, ranks must be decided relative to a stated
scale and tolerance. For each matrix $X$ whose rank is required, one may
count singular values above
$\tau_{\mathrm{abs}}+\tau_{\mathrm{rel}}\lVert X\rVert$, with tolerances
chosen for the data scale and precision. Singular values near this threshold
require a sensitivity check; the orthogonal implementation does not remove
the nonrobustness in \cref{rem:nonrobustness}.

The coefficient matrices, the current orthogonal state basis, and the work
arrays require $O\bigl(mn(n+r)\bigr)$ storage. Indeed, $p\leq r+n$ and
$h\leq mn$; the tall branch does not store a full $p\times p$ factor,
while in the wide branch $p^2<ph$. The factorization of $M$ costs at
most $O\bigl(ph\min(p,h)\bigr)\leq O(ph^2)$ arithmetic operations.
Including matrix products, the decomposition of $R$, and
reorthogonalization, each stage costs at most
$O\bigl(m^2n^2(n+r)\bigr)$. At most $n$ stages and the final
singular-value computation for $\mathcal D$ therefore cost
$O\bigl(m^2n^3(n+r)\bigr)$ in dense real arithmetic. In contrast, at
$n=14$, $r=2$, and $m=1$, the prescribed matrix $\mathsf N_n$ alone
would require about $56$ GiB of dense storage at eight bytes per entry,
by the dimensions in \cref{eq:pair}.

\subsection{Input configurations and control energy}\label{sec:numerical-example}

We compare the energy required to reduce the terminal error under three
input configurations of the same two-dimensional system with two Brownian
motions. Take
$A=\bigl(\begin{smallmatrix}0&1\\0&0\end{smallmatrix}\bigr)$,
$C_1=\bigl(\begin{smallmatrix}1/2&0\\1&1/2\end{smallmatrix}\bigr)$ and
$C_2=\operatorname{diag}(1/3,-1/3)$.
The five candidate input columns are given by
$B=e_2(1,1,0,0,0)$,
$D_1=[\,0\ e_2\ e_1\ e_2\ 0\,]$ and
$D_2=[\,0\ e_1\ 0\ 0\ e_2\,]$, where $e_1,e_2$ are the
coordinate vectors and $0$ denotes a zero column.
Configurations I, II and III use columns $\{2\}$, $\{1,2\}$ and
$\{1,2,3,4,5\}$, respectively, in all three input matrices.

Orthogonal elimination and exact rational rank calculations give
\cref{tab:example-classification}. Adding the drift-only input in II
removes the geometric obstruction, while the additional diffusion inputs
in III make $\mathcal D$ surjective. Thus II permits exact null control,
and III also permits exact control to arbitrary random targets.

\begin{table}[tb]
\centering
\caption{Input classification by \cref{cor:rank-control}.}
\label{tab:example-classification}
\begin{tabular}{@{}cccc@{}}
\toprule
Case & $(\delta_1,\delta_2)$ & $\operatorname{rank}\mathcal D$ & Properties \\
\midrule
I & $(0,0)$ & $1$ & None \\
II & $(1,2)$ & $1$ & AC, ANC, ENC \\
III & $(1,2)$ & $4$ & AC, ANC, ENC, EC \\
\bottomrule
\end{tabular}

\end{table}

For $T=1$ and $y_0=e_1$, we solve
\begin{align*}
\min_{u\in\mathcal U_1}\Bigl\{
\mathbb E\int_0^1|u(t)|^2\,dt+\lambda\mathbb E|y(1)|^2\Bigr\},
\end{align*}
with $\lambda=10^k$, $k=0,\ldots,6$.
The Riccati method in \cref{sec:strict-alternative}, with drift $A$,
zero running state weight and terminal weight $\lambda I_2$, gives the
optimal feedback $u_\lambda=K_\lambda y_\lambda$.
We compute $M_\lambda(t)=\mathbb E[y_\lambda(t)y_\lambda(t)^\top]$ from
\begin{align*}
\begin{cases}
\displaystyle \dot M_\lambda=F_\lambda M_\lambda+M_\lambda F_\lambda^\top
+\sum_{j=1}^2G_{\lambda,j}M_\lambda G_{\lambda,j}^\top,\\[4pt]
M_\lambda(0)=e_1e_1^\top,
\end{cases}
\end{align*}
where $F_\lambda=A+BK_\lambda$ and
$G_{\lambda,j}=C_j+D_jK_\lambda$.
The terminal error is $\varepsilon_\lambda=\operatorname{tr}M_\lambda(1)$
and the control energy is
\begin{align*}
E_\lambda=\int_0^1\operatorname{tr}
\bigl(K_\lambda(t)M_\lambda(t)K_\lambda(t)^\top\bigr)\,dt.
\end{align*}

\begin{figure}[tb]
\centering
\includegraphics[width=0.65\linewidth]{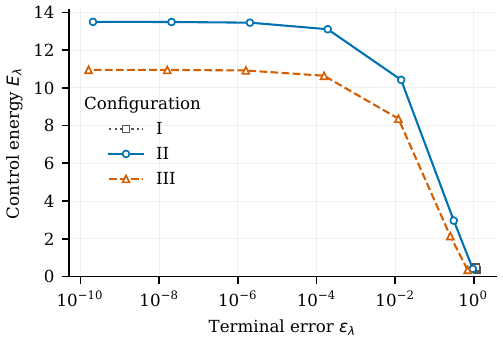}
\caption{Terminal error and expected control energy for
$\lambda=10^k$, $k=0,\ldots,6$, computed from the Riccati and
second-moment equations. Markers show all seven penalties in each
configuration.}
\label{fig:example-feedback}
\end{figure}

\Cref{fig:example-feedback} shows a persistent terminal error in I,
whereas II and III attain decreasing errors with energies that level off.
In II, increasing $\lambda$ from $10^2$ to $10^6$ reduces the error
from $1.4105\times10^{-2}$ to $2.0745\times10^{-10}$, while the energy
increases from $10.426$ to $13.503$.
At $\lambda=10^6$, III gives an error of $1.6017\times10^{-10}$
with energy $10.959$, while I retains an error of about $1.1293$.
The behavior in II and III agrees with the bounded-energy approximation
provided by ENC. Their different EC classifications are determined by
the joint diffusion rank in \cref{tab:example-classification}.

\section{Conclusions}\label{sec:conclusion}
We have obtained a unified finite algebraic classification of four
controllability properties for constant-coefficient linear stochastic
systems with multidimensional Brownian noise and arbitrary joint
diffusion rank. A common geometric obstruction characterizes
approximate controllability, approximate null controllability and
exact null controllability, while exact controllability additionally
requires surjective diffusion control. The feedback construction
provides finite-energy null control at any prescribed positive time,
and orthogonal elimination makes the rank tests computable without
forming exponentially large block matrices. The numerical example
relates this classification to the terminal-error and energy behavior
of penalized LQ controls under successive additions of input channels.
Quantifying the effects of bounded gains
and finite switching resolution is a natural direction for further work.

\section*{Acknowledgements}
Qi L\"u was supported by the National Natural Science Foundation of China
under grant 12625102. Yu Wang was supported by the National Natural Science
Foundation of China under grant 12401589 and the Sichuan Science and
Technology Program under grant 2026NSFSC0777.

\bibliographystyle{siam-paper-no}
\bibliography{references}

\end{document}